\documentclass[12pt]{article}

\title{Distance Powers and Distance Matrices of Integral Cayley Graphs over Abelian Groups}

\author{Walter Klotz\\
{\small Institut f\"ur Mathematik}\\[-0.8ex]
{\small Technische Universit\"at Clausthal, Germany}\\[-0.8ex]
{\small \texttt{klotz@math.tu-clausthal.de}}\\
Torsten Sander\\
\small Fakult\"at f\"ur Informatik\\[-0.8ex]
\small Ostfalia Hochschule f\"ur angewandte Wissenschaften, Germany\\[-0.8ex]
\small \texttt{t.sander@ostfalia.de}
}

\date{May 21, 2012\\
\small  Mathematics Subject Classification: 05C25, 05C50}

\usepackage{amsmath}                                                            
\usepackage{amssymb}                                                            
\usepackage{amsfonts}
\usepackage{amsthm}   
\usepackage{centernot}

\newtheorem{theorem}{Theorem}
\newtheorem{lemma}{Lemma}
\newtheorem{corollary}{Corollary}
\newtheorem{proposition}{Proposition}                                                       

\newcommand{\Atom}{\operatorname*{Atom}}
\newcommand{\ord}{\operatorname{ord}}
\newcommand{\lcm}{\operatorname{lcm}}
\newcommand{\Cay}{\operatorname{Cay}}
\newcommand{\DM}{\operatorname{DM}}

\newcommand{\Q}{\mathbb{Q}}
\newcommand{\C}{\mathbb{C}}
\newcommand{\Z}{\mathbb{Z}}

\newenvironment{casemarker}%
{\noindent\ignorespaces}%
{\par\noindent%
\ignorespacesafterend}

\begin{document}

\maketitle

\begin{abstract}
It is shown that distance powers of an integral Cayley graph over an abelian group $\Gamma$
are again integral Cayley graphs over $\Gamma$. Moreover, it is proved that distance
matrices of integral Cayley graphs over abelian groups have integral spectrum.
\end{abstract}

%-------------------------------------------------------------------------------------------------------

\section{Introduction}

Eigenvalues of an undirected graph $G$ are the eigenvalues of an arbitrary adjacency matrix of $G$.
General facts about graph spectra can e.g.~be found in \cite{cvetkovic} or \cite{vandam}.
Harary and Schwenk \cite{harary} defined $G$ to be {\em integral} if all of its eigenvalues
are integers. For a survey of integral graphs see \cite{balinska}.
In \cite{ahmadi} the number of integral graphs on $n$ vertices is estimated.
Known characterizations of integral graphs are restricted to certain graph classes, see e.g.~%
\cite{abdollahi}, \cite{klotz2}, or \cite{so}. 
Here we concentrate on integral Cayley graphs over abelian groups and their distance powers.

Let $\Gamma$ be a finite, additive group, $S \subseteq \Gamma,~-S=\{-s:~s\in S\}=S$.
The undirected {\em Cayley graph over $\Gamma$ with shift set (or symbol) $S$}, $\Cay(\Gamma,S)$, 
has vertex set $\Gamma$. Vertices $a,~b\in \Gamma$ are adjacent if and only if $a-b \in S$.
For general properties of Cayley graphs we refer to Godsil and Royle \cite{godsil} or Biggs \cite{biggs}.
Note that $0\in S$ generates a loop at every vertex of $\Cay(\Gamma,S)$. Many definitions of
Cayley graphs exclude this case, but its inclusion saves us from sacrificing clarity of
presentation later on.

In our paper \cite{klotz1} we proved for an abelian group $\Gamma$ that Cay$(\Gamma,S)$ is integral
if $S$ belongs to the Boolean algebra $B(\Gamma)$ generated by the subgroups of $\Gamma$.
Our conjecture that the converse is true for all integral Cayley graphs over abelian groups has recently 
been proved by Alperin and Peterson \cite{alperin}.

Let $G=(V,E)$ be an undirected graph with vertex set $V$ and edge set $E$, $D$ a finite set of 
nonnegative integers. The {\em distance power} $G^D$ of $G$ is an undirected graph with vertex set $V$.
Vertices $x$ and $y$ are adjacent in $G^D$, if their distance $d(x,y)$ in $G$ belongs to $D$.
We prove that if $G$ is an integral Cayley graph over the abelian group $\Gamma$, then every
distance power $G^D$ is also an integral Cayley graph over $\Gamma$. Moreover, we show that
in a very general sense distance matrices of integral Cayley graphs over abelian groups have 
integral spectrum. This extends an analogous result of Ili\'{c} \cite{ilic} for integral circulant 
graphs, which are the integral Cayley graphs over cyclic groups. Finally, we show that the class of gcd-graphs,
another subclass of integral Cayley graphs over abelian groups (see \cite{klotz2}), is also closed
under distance power operations.

In our proofs we apply the mentioned result and techniques of Alperin and Peterson.
To make this paper more selfcontained and to draw additional attention to this beautiful result,
we include a proof reduced to our purposes on a more combinatorial level. 

%-------------------------------------------------------------------------------------------------------

\section{The Boolean Algebra $B(\Gamma)$}

Let $\Gamma$ be an arbitrary finite, additive group. We collect facts about the 
Boolean algebra $B(\Gamma)$ generated by the subgroups of $\Gamma$. 

\subsection{Atoms of $B(\Gamma)$}

Let us determine the second minimal elements of $B(\Gamma)$. To this end, we consider elements of $\Gamma$
to be equivalent, if they generate the same cyclic subgroup. The equivalence classes of this relation
partition $\Gamma$ into nonempty disjoint subsets. We shall call these sets {\em atoms}.
The atom represented by $a\in \Gamma$, $\Atom(a)$, consists of the generating elements 
of the cyclic group $\left<a\right>$.
\[
\begin{split}
  \Atom(a) &= \{ b\in\Gamma:~\left< a\right> = \left< b\right>\} \\
           &= \{ ka:~k\in \Z,~1\leq k\leq \ord_\Gamma(a),~ \gcd(k, \ord_\Gamma(a))=1\}.
\end{split}
\]
Here, $\Z$ stands for the set of all integers. For a positive integer $k$ and $a\in \Gamma$ we denote 
as usual by $ka$ the $k$-fold sum of terms $a$, $(-k)a=-(ka),~0a=0$. By $\ord_\Gamma(a)$ 
we mean the order of $a$ in $\Gamma$.

Each set $\Atom(a)$ can be obtained  by removing from $\left<a\right>$ 
all elements of its proper subgroups. We bear in mind that every set $S\in B(\Gamma)$ 
can be derived from the cyclic subgroups of $\Gamma$ by means of repeated union, intersection and
complement (with respect to $\Gamma$). Thus we easily arrive at the following propositions.

\begin{proposition}\label{prop1a}
For every $a\in \Gamma$ and every $S\in B(\Gamma)$ holds:
$$\Atom(a) \subseteq S \text{ or } \Atom(a)\subseteq \bar{S}=\Gamma \setminus S.$$
\end{proposition}

\begin{proposition}\label{prop1b}
For $a\in \Gamma$ let 
$\{U_1,\ldots ,U_k\}$ be the family of proper subgroups of $\left<a\right>$.
Then we have
$$\Atom(a)~=~\left<a\right>\setminus ( U_1 \cup \ldots \cup U_k).$$
\end{proposition}

\begin{proposition}\label{prop1c}
For an arbitrary finite group $\Gamma$ the following statements are true:
\begin{enumerate}
\item $\Atom(a) \in B(\Gamma)$ for every $a\in \Gamma$.
\item For no $a\in \Gamma$ there exists a nonempty proper subset of $\Atom(a)$ that belongs to $B(\Gamma)$.
\item Every nonempty set $S\in B(\Gamma)$ is the union of some sets $\Atom(a),~a\in \Gamma$.
\end{enumerate}
\end{proposition}

%--------------------------------------------------------

\subsection{Sums of Sets in $B(\Gamma)$}

In this subsection $\Gamma$ denotes a finite, additive, abelian group.
We define the sum of nonempty subsets $S,~T$ of $\Gamma$:
$$S+T~=~\{s+t:~s\in S,~t\in T\}.$$
We are going to show that the sum of sets in $B(\Gamma)$ is again a set in $B(\Gamma)$.

\begin{lemma}\label{lem:atomsum}
If $\Gamma$ is a finite abelian group and $a,~b\in\Gamma$ then
$$\Atom(a)+\Atom(b)~\in~B(\Gamma).$$
\end{lemma}

\begin{proof}
We know that $\Gamma$ can be represented (see Cohn \cite{cohn}) as a direct sum of cyclic groups 
of prime power order. This can be grouped as
$$\Gamma = \Gamma_1 \oplus \Gamma_2 \oplus \cdots \oplus \Gamma_r,$$
where $\Gamma_i$ is a direct sum of cyclic groups, the order of which is a power of a prime $p_i$,
$|\Gamma_i|=p_i^{\alpha_i}$, $\alpha_i\geq 1$ for $i=1,\ldots,r$ and $p_i\neq p_j$ for $i\neq j$. 
Hence we can write each element $x\in \Gamma$ as an $r$-tuple $(x_i)$ with
$x_i\in \Gamma_i$ for $i=1,\ldots,r$.

The order of $x_i\in \Gamma_i$, $\ord_{\Gamma_i}(x_i)$, is a divisor of $p_i^{\alpha_i}$.
Therefore, integer factors in the $i$-th coordinate of $x$ may be reduced modulo $p_i^{\alpha_i}$.
The order of $x\in \Gamma$, $\ord_\Gamma(x)$, is the least common multiple of the orders of its
coordinates:
\begin{equation}\label{eq1}
\ord_\Gamma(x)~=~\lcm(\ord_{\Gamma_1}(x_1),\ldots,\ord_{\Gamma_r}(x_r)).
\end{equation}
This implies that all prime divisors of $\ord_\Gamma(x)$ belong to $\{p_1,\ldots,p_r\}$.

Let $a=(a_i),~b=(b_i)$ be elements of $\Gamma$. The statement of the lemma becomes trivial for
$a=0$ or $b=0$. So we may assume $a\neq 0$ and $b\neq 0$. An arbitrary element $w\in \Atom(a)+\Atom(b)$
has the following form:
\begin{equation}\label{eq2}
\begin{split}
&w~=~\lambda a+\mu b,\\
&1\leq\lambda \leq \ord_\Gamma(a),~~\gcd(\lambda,\ord_\Gamma(a))=1,\\
&1\leq\mu \leq \ord_\Gamma(b),~~\gcd(\mu,\ord_\Gamma(b))=1. 
\end{split}
\end{equation}
We have to show $\Atom(w) \subseteq \Atom(a) + \Atom(b)$. To this end, we choose the integer $\nu$ with
$1\leq \nu \leq \ord_\Gamma(w)$, $\gcd(\nu,\ord_\Gamma(w))=1$, and show $\nu w \in \Atom(a)+\Atom(b)$.\\

\begin{casemarker}
Case 1.~~ $(p_1p_2\cdots p_r) \mid \ord_\Gamma(w)$.\\
\end{casemarker}

By $\gcd(\nu,\ord_\Gamma(w))=1$ we know that $\nu$ has no prime divisor in $\{p_1,\ldots,p_r\}$.
On the other hand all prime divisors of $\ord_\Gamma(a)$ and of $\ord_\Gamma(b)$ are in $\{p_1,\ldots,p_r\}$.
This implies $\gcd(\nu,\ord_\Gamma(a))=1$ and $\gcd(\nu,\ord_\Gamma(b))=1$. Setting $\lambda'=\nu \lambda$
and $\mu'=\nu \mu$ we achieve
\[\begin{split}
&\gcd(\lambda',\ord_\Gamma(a))=1,~\lambda' a\in \Atom(a),\\
&\gcd(\mu',\ord_\Gamma(b))=1,~\mu' b\in \Atom(b).
\end{split}
\]
Now we have by (2):
$$\nu w = \nu \lambda a +\nu \mu b = \lambda' a + \mu' b \in \Atom(a)+\Atom(b).$$

\begin{casemarker}
Case 2.~~ $(p_1p_2\cdots p_r) \centernot\mid \ord_\Gamma(w)$.\\
\end{casemarker}

Trivially, for $w=0\in \Atom(a)+\Atom(b)$ we have $\nu w \in \Atom(a)+\Atom(b)$. 
Therefore, we may assume $w\neq 0$. Without loss of generality let
\begin{equation}\label{eq3}
(p_1\cdots p_k)\mid \ord_\Gamma(w),~\gcd(\ord_\Gamma(w),p_{k+1}\cdots p_r)=1,~1\leq k<r.
\end{equation}
Now (\ref{eq1}) and (\ref{eq3}) imply
\begin{equation}\label{eq4}
\begin{split}
w~=~\lambda a +\mu b ~
&=~(\lambda a_1+\mu b_1,\ldots,\lambda a_k+\mu b_k,0,\ldots,0),\\
&\lambda a_i+\mu b_i\neq 0 \text{ for }i=1,\ldots,k.
\end{split}
\end{equation}
By $\gcd(\nu,ord_\Gamma(w))=1$ we know $\gcd(\nu,p_1\cdots p_k)=1$. If even more
 $\gcd(\nu,p_1\cdots p_r)=1$ then we deduce $\nu w\in \Atom(a)+\Atom(b)$ as in Case 1.
So we may assume that $\nu$ has at least one prime divisor in $\{p_{k+1},\ldots,p_r\}$.
Without loss of generality let
$$\gcd(\nu,p_1\cdots p_l)=1,~(p_{l+1}\cdots p_r)\mid \nu,~k\leq l <r.$$
We define
\begin{equation}\label{eq5}
\nu'~=~\nu~+~p_1^{\alpha_1}\cdots p_l^{\alpha_l}.
\end{equation}
If we observe that integer factors in the $i$-th coordinate of $w$ can be reduced
modulo $p_i^{\alpha_i}$, then we see by (\ref{eq4}): $\nu'w=\nu w$.
Moreover, (\ref{eq5}) and the properties of $\nu$ imply $\gcd(\nu',p_1\cdots p_r)=1$.
As in Case 1 we now conclude
$\nu w~=~\nu' w~\in~\Atom(a)+\Atom(b).$
\end{proof}

\begin{theorem}\label{thm1}
If $\Gamma$ is a finite abelian group with nonempty subsets $S, T\in B(\Gamma)$ then $S+T\in B(\Gamma)$.
\end{theorem}

\begin{proof}
The sets $S$ and $T$ are unions of atoms of $B(\Gamma)$.
$$S~=~\bigcup_{i=1}^k ~\Atom(a_i),~~T~=~\bigcup_{j=1}^l ~\Atom(b_j).$$
Then we have
\begin{equation}\label{eq6}
S+T~=~\bigcup_{1\leq i\leq k, 1\leq j\leq l}~(\Atom(a_i)+\Atom(b_j)).
\end{equation}
According to Lemma \ref{lem:atomsum} the sum $\Atom(a_i)+\Atom(b_j)$ is an element of $B(\Gamma)$. Therefore,
(\ref{eq6}) implies $S+T \in B(\Gamma)$.
\end{proof}

%-------------------------------------------------------------------------------------------------------
\section{Integral Subsets and Group Characters}
%-------------------------------------------------------------------------------------------------------

Let $\Gamma$ be a finite, additive group, $f:\Gamma \rightarrow \C$ a complex valued 
function on $\Gamma$. A subset $S\subseteq \Gamma$ is called 
{\em $f$-integral}, cf.~our paper \cite{klotz1}, if
$$f(S)~=~\sum_{s\in S}~f(s)~\in \Z~.$$
We agree upon $f(\emptyset)=0$. So the empty set is always $f$-integral.

\begin{lemma}\label{lem:fint}
If all cyclic subgroups of the finite group $\Gamma$ are $f$-integral, then every set $S\in B(\Gamma)$
is $f$-integral.
\end{lemma}

\begin{proof}
Suppose that $S\in B(\Gamma),~S\neq \emptyset$. By Proposition \ref{prop1c}, $S$ is the disjoint union of
atoms $A_1,\ldots,A_r$ of $B(\Gamma)$. Then we have $f(S)=f(A_1)+\ldots +f(A_r)$. Therefore,
it is sufficient to show that every atom is $f$-integral.
According to Proposition \ref{prop1b}, every atom $A$ with $a\in A$ has a representation
$$A~=~\left<a\right>\setminus ( U_1 \cup \ldots \cup U_k)$$
with certain cyclic subgroups $U_1,\ldots,U_k$ of $\left<a\right>$. Hence,
$$f(A)~=~f(\left<a\right>) - f( U_1 \cup \ldots \cup U_k).$$
As $f(\left<a\right>)\in \Z$, we may concentrate on $f( U_1 \cup \ldots \cup U_k)$, which can be evaluated by the
principle of inclusion and exclusion (see e.g. \cite{vanLint}).
\begin{equation}\label{eqinex}
f( U_1 \cup \ldots \cup U_k)~=~\sum_{p=1}^k(-1)^{p-1} \sum_{1\leq j_1<\ldots <j_p\leq k}~f(U_{j_1}\cap \ldots \cap U_{j_p})
\end{equation}
Since $U_{j_1}\cap \ldots \cap U_{j_p}$ is a cyclic subgroup of $\Gamma$, all terms in (\ref{eqinex}) are integers.
Hence the claim follows.
\end{proof}

Let $\Gamma$ be a finite additive group with $n$ elements, $|\Gamma|=n$.
A {\em character} $\psi$ of $\Gamma$ is a homomorphism from $\Gamma$ into the multiplicative
group of complex numbers, $\psi:\Gamma \rightarrow \C\backslash \{0\}$, 
$$\psi(\mu a + \nu b)~=~(\psi(a))^\mu~(\psi(b))^\nu \text{ for every }a,~b \in \Gamma
\text{ and }\mu,~\nu \in Z.$$
Fermat's little theorem yields
$$(\psi(a))^n~=~\psi(na)~=~\psi(0)~=~1.$$
Therefore, $\psi(a)$ is an $n$-th root of unity for every $a\in \Gamma$.

\begin{lemma}\label{lempsiH}
Let $H$ be a subgroup of $\Gamma$ and $\psi$ a character of $\Gamma$.
If $H$ contains an element $g$ with $\psi(g)\neq 1$, then $\psi(H)=0$ else $\psi(H)=|H|$.
\end{lemma}

\begin{proof}
If $g\in H$ and $\psi(g)\neq 1$, then we have
$$\psi(H)~=~\sum_{h\in H}~\psi(h+g)~=~\psi(g) \psi(H)$$
so that
$$(1-\psi(g))\psi(H)~=~0,$$
which implies $\psi(H)=0$. If $\psi(g)=1$ for every $g\in H$ then $\psi(H)=|H|$.
\end{proof}

\begin{corollary}\label{cor1}
For an arbitrary character $\psi$ of $\Gamma$ every set $S\in B(\Gamma)$ is $\psi$-integral.
\end{corollary}

\begin{proof}
According to Lemma \ref{lempsiH} every subgroup $H$ of $\Gamma$ is $\psi$-integral.
Now Lemma \ref{lem:fint} implies that every set $S\in B(\Gamma)$ is $\psi$-integral.
\end{proof}

%-------------------------------------------------------------------------------------------------------
\section{Eigenvalues and Eigenvectors of Cayley graphs}
%-------------------------------------------------------------------------------------------------------

First we show that the characters of a finite group $\Gamma$ represent eigenvectors
for every Cayley graph over $\Gamma$.

\begin{lemma}\label{lemeigen}
Let $\psi$ be a character of the additive group $\Gamma = \{v_1,\ldots,v_n\},~S\subseteq \Gamma,~
-S=S$. Assume that $A=(a_{i,j})$ is the adjacency matrix of $G=\Cay(\Gamma,S)$ with
respect to the given ordering of the vertex set $V(G)=\Gamma$. Then the vector
$(\psi(v_j))_{j=1,\ldots,n}$ is an eigenvector of $A$ with eigenvalue $\psi(S)$.
\end{lemma}

\begin{proof}
We evaluate the product of the $i$-th row of $A$ with $(\psi(v_j))_{j=1,\ldots,n}$:
\[ 
\sum_{j=1}^n~a_{i,j}\psi(v_j) = \sum_{1\leq j\leq n,~v_j-v_i\in S}~\psi(v_j) = \sum_{s\in S}~\psi(s+v_i)
= \psi(v_i) \psi(S).
\]
\end{proof}
If we refer to the characters of $\Gamma$ as eigenvectors of an arbitrary Cayley graph over
$\Gamma$, this is meant in the sense of Lemma \ref{lemeigen}.
From now on we assume that the finite additive group $\Gamma$ is abelian, $|\Gamma|=n$.
We sketch as in \cite{klotz1}, why $\Gamma$ has exactly $n$ pairwise orthogonal characters.

For an integer $m\geq 1$ we denote by $Z_m$ the additive group of integers modulo $m$,
the ring of integers modulo $m$, or simply the set $\{0,1,\ldots,m-1\}$. The particular choice
will be clear from the context. Let $\Gamma$ be represented as the direct sum of cyclic groups,  
\begin{equation}\label{eq8}
\Gamma~=~Z_{n_1}\oplus \cdots \oplus Z_{n_k},~|\Gamma|=n=n_1\cdots n_k.
\end{equation}
The elements $x\in \Gamma$ are considered as elements of the Cartesian product 
$Z_{n_1}\times \cdots \times Z_{n_k}$, 
$$x=(x_i)=(x_1,\ldots,x_k),~x_i\in Z_{n_i}=\{0,1,\ldots,n_i-1\},~1\leq i\leq k.$$
Addition is coordinatewise modulo $n_i$. 
Denote by $e_i$ the unit vector with entry 1 in position $i$ and entry 0 in all positions $j\neq i$.
A character $\psi$ of $\Gamma$ is uniquely determined by its values $\psi(e_i),~1\leq i\leq k$:
\begin{equation}\label{eq9}
x=(x_i)~=~\sum_{i=1}^k~x_ie_i,~~\psi(x)~=~\prod_{i=1}^k~(\psi(e_i))^{x_i}.
\end{equation}
As $e_i\in \Gamma$ has order $n_i$, the value $\psi(e_i)$ must be a complex $n_i$-th root of unity.
So there are $n_i$ possible choices for the value of $\psi(e_i)$. Let $\zeta_i$ be a primitive 
$n_i$-th root of unity for every $i,~1\leq i\leq k$. For every $\alpha = (\alpha_i)\in \Gamma$ a
character $\psi_\alpha$ can be uniquely defined by 
\begin{equation}\label{eq10}
\psi_\alpha(e_i)~=~\zeta_i^{\alpha_i},~1\leq i \leq k.
\end{equation}
Thus all $|\Gamma|=n$ characters of the abelian group $\Gamma$ can be obtained.

\begin{proposition}\label{prop3}
Let $\psi_1,\ldots,\psi_n$ be the distinct characters of the additive abelian group 
$\Gamma = \{v_1,\ldots ,v_n\},~S\subseteq \Gamma,~-S=S$. Assume that $A=(a_{i,j})$
is the adjacency matrix of $G=\Cay(\Gamma,S)$ with respect to the given ordering of the vertex set
$V(G)=\Gamma$. Then the vectors $(\psi_i(v_j))_{j=1,\ldots ,n},~1\leq i\leq n$, constitute
an orthogonal basis of $\C^n$ consisting of eigenvectors of $A$. To the eigenvector
 $(\psi_i(v_j))_{j=1,\ldots ,n}$ belongs the eigenvalue $\psi_i(S)$.
\end{proposition}

\begin{proof}
By Lemma \ref{lemeigen} and the considerations above it remains to prove that for 
$\alpha=(\alpha_i)\in \Gamma,~\beta=(\beta_i)\in \Gamma,~\alpha \neq \beta$, the eigenvectors
$(\psi_\alpha(v_j))_{j=1,\ldots ,n}$ and $(\psi_\beta(v_j))_{j=1,\ldots ,n}$ are orthogonal
(with respect to the standard inner product of $\C^n$).
We represent $\Gamma$ by (\ref{eq8}) and define $\psi_\alpha$ and $\psi_\beta$ according to 
(\ref{eq9}) and (\ref{eq10}). A straightforward calculation verifies that
\begin{equation}\label{eq11}
\sigma = \sum_{j=1}^n~\psi_\alpha(v_j)\overline{\psi_\beta(v_j)}
= \prod_{i=1}^k~\sum_{0\leq x_i<n_i}~\zeta_i^{(\alpha_i-\beta_i)x_i}.
\end{equation}
As $\alpha \neq \beta$ we may assume e.g. $\alpha_1 \neq \beta_1$.  Then
$$\sum_{0\leq x_1<n_1}~\zeta_1^{(\alpha_1-\beta_1)x_1}~
=~\frac{\zeta_1^{(\alpha_1-\beta_1)n_1}-1}{\zeta_1^{(\alpha_1-\beta_1)}-1}~=~0 $$
implies $\sigma = 0$ by (\ref{eq11}).
\end{proof}

\begin{corollary}\label{cor2}
Let $\psi_\alpha$ and $\psi_\beta$ be characters of the abelian group $\Gamma = \{v_1,\ldots ,v_n\}$.
Then we have
$$\sum_{j=1}^n~\psi_\alpha(v_j)\overline{\psi_\beta(v_j)}~=~\left\{
\begin{array}{l}
0\text{ for }\psi_\alpha \neq \psi_\beta\\
n\text{ for }\psi_\alpha = \psi_\beta .
\end{array}
\right. $$
\end{corollary}

\begin{corollary}\label{cor3}
Let $\Gamma$ be a finite abelian group. For every set $S\in B(\Gamma)$ the Cayley graph
$\Cay(\Gamma,S)$ is integral.
\end{corollary}

\begin{proof}
According to Proposition \ref{prop3} all eigenvalues of $\Cay(\Gamma,S)$ have the form $\psi(S)$ with a character
$\psi$ of $\Gamma$. By Corollary \ref{cor1} we know that $\psi(S)$ is integral for every $S\in B(\Gamma)$.
\end{proof}

We are going to prove the converse of Corollary \ref{cor3}.
As before, $\Gamma=\{v_1,\ldots,v_n\}$ denotes an abelian group with characters $\psi_1,\ldots,\psi_n$.
The {\em characteristic vector} $\chi_S$ of $S\subseteq \Gamma$ is defined by
$$\chi_S~=~(\chi_S(v_j)),~~\chi_S(v_j)~=~\left\{
\begin{array}{l}
1\text{, if }v_j\in S\\
0\text{, if }v_j\not\in S
\end{array} \right.~\text{ for }1\leq j\leq n.$$
The {\em character matrix} $H=(h_{i,j})$ with respect to the ordering $v_1,\ldots,v_n$ of the elements
of $\Gamma$ and the ordering $\psi_1,\ldots,\psi_n$ of the characters of $\Gamma$ is defined by
\begin{equation}\label{eq12}
h_{i,j}~=~\psi_i(v_j)\text{ for }i,j\in \{1,\ldots,n\}.
\end{equation}
Corollary \ref{cor2} implies
\begin{equation}\label{eq13}
\begin{split}
H\overline{H}^T~&=~nI_n,\\
H^{-1}~&=~\frac{1}{n} \overline{H}^T~=~\frac{1}{n}(h_{j,i}^{-1}).
\end{split}
\end{equation}
Here $I_n$ is the $n\times n$ unit matrix and $\overline{H}^T$ denotes the transpose of the
complex conjugate $\overline{H}=(\overline{h_{i,j}})$. Observe that 
$\overline{h_{i,j}}=h_{i,j}^{-1}$, because $h_{i,j}$ is an $n$-th root of unity.

\begin{lemma}\label{lemS}
Let $\Gamma$ be a finite abelian group, $S\subseteq \Gamma$, $S=-S$, $S\neq \emptyset$. 
Assume that the Cayley graph $\Cay(\Gamma,S)$ is integral.
Then every atom of $B(\Gamma)$ is either a subset of $S$ or disjoint from $S$.
\end{lemma}

\begin{proof}
Let $w=(w_i)$ denote the vector resulting from multiplication of the character matrix $H=(h_{i,j})$ 
defined by (\ref{eq12}) with the characteristic vector $\chi_S$ of $S$. Then, for $i=1,\ldots,n$, we have
$w_i=\psi_i(S)$.
According to Proposition \ref{prop3}, the entries $w_i$ of $w$ are the eigenvalues of $G=\Cay(\Gamma,S)$.
If $G$ is integral, then all entries of $w$ are integers. Using (\ref{eq13}) we solve $H\chi_S=w$ for $\chi_S$ 
and obtain
$$\chi_S~=~\frac{1}{n}\overline{H}^Tw.$$
 For an arbitrary vertex $v_k\in \Gamma=\{v_1,\ldots,v_n\}$ we have 
\begin{equation}\label{eq14}
\chi_S(v_k)~=~~\frac{1}{n}\sum_{i=1}^n(\psi_i(v_k))^{-1}w_i.
\end{equation}
Let $v_q \in \Atom(v_k)$. In order to prove the Lemma we are going to show that
$\chi_S(v_q)=\chi_S(v_k)$.
By the choice of $v_q$, we have $\left<v_q\right>=\left<v_k\right>$ and 
$\ord_\Gamma(v_q)=\ord_\Gamma(v_k)=m$ for a divisor $m$ of $n$.
This implies $v_q = r v_k$ for some $r\in \{1,\ldots,m\}$ with $\gcd(r,m)=1.$
It follows from (\ref{eq14}) that 
\begin{equation}\label{eq15}
\chi_S(v_q)~=~~\frac{1}{n}\sum_{i=1}^n(\psi_i(v_k))^{-r}w_i.
\end{equation}
Since $\ord_\Gamma(v_k)=m$ we see that
$$(\psi_i(v_k))^m=\psi_i(mv_k)=\psi_i(0)=1,$$
which means that $\psi_i(v_k)$ is an $m$-th root of unity for every $i=1,\ldots,n$.
If $\xi$ is a primitive $m$-th root of unity, then equation (\ref{eq14}) is an equation 
in the field $\Q(\xi)$ over the rationals $\Q$. As $\gcd(r,m)=1$,
we can uniquely define an automorphism $F$ of $\Q(\xi)$ by $F(\xi)=\xi^r$.
The $m$-th root of unity $\psi_i(v_k)$ is a power of $\xi$, therefore 
$$F(\psi_i(v_k))=(\psi_i(v_k))^r\text{ for }i=1,\ldots,n.$$ 
Moreover, the automorphism $F$ leaves all elements of $\Q$ unchanged. Applying $F$
to (\ref{eq14}) and observing (\ref{eq15}) we achieve
$$\chi_S(v_k)=F(\chi_S(v_k))=\frac{1}{n} \sum_{i=1}^n (\psi_i(v_k))^{-r}w_i = \chi_S(v_q).$$
\end{proof}
 
We can now confirm the result of Alperin and Peterson \cite{alperin}.

\begin{theorem}\label{thm2}
Let $\Gamma$ be a finite abelian group, $S\subseteq \Gamma,~S=-S$.
Then the Cayley graph $\Cay(\Gamma,S)$ is integral if and only if $S\in B(\Gamma)$.
\end{theorem}

\begin{proof}
If $S\in B(\Gamma)$, then $\Cay(\Gamma,S)$ is integral by Corollary \ref{cor3}.
To prove the converse let $\Cay(\Gamma,S)$ be integral. We may assume $S\neq \emptyset$.
Then we see by Lemma \ref{lemS} that every atom of $B(\Gamma)$ is either a subset of $S$
or is disjoint to $S$. This implies that $S$ is the union of atoms and therefore it
belongs to $B(\Gamma)$. 
\end{proof}

%-------------------------------------------------------------------------------------------------------
\section{Distance Powers and Distance Matrices}
%-------------------------------------------------------------------------------------------------------

We repeat the definition of the distance power $G^D$ of an undirected graph $G=(V,E)$
from the Introduction. Let $D$ be a set of nonnegative integers. The distance power $G^D$ has vertex set $V$.
Vertices $x,~y$ are adjacent in $G^D$, if their distance in $G$ is $d(x,y)\in D$.
If $G$ is not connected, it makes sense to allow $\infty \in D$. Clearly, $G^\emptyset$ is the graph
without edges on $V$. The edge set of $G^{\{0\}}$ consists of a single loop at every vertex of $G$.
If $G$ has no loops then $G^{\{1\}} = G$.

\begin{theorem}\label{thm3}
If $G=\Cay(\Gamma,S)$ is an integral Cayley graph over the finite abelian group $\Gamma$
and if $D$ is a set of nonnegative integers (possibly including $\infty$), 
then the distance power $G^D$ is also an integral Cayley graph over $\Gamma$. 
\end{theorem}

\begin{proof}
If $D=\emptyset$ then $G^D=\Cay(\Gamma,\emptyset)$ is an integral Cayley graph over $\Gamma$.
We now consider the case, where $D$ has only one element,
$$D = \{d\},~~d\in \{0,1,\ldots,\infty\}.$$
In several steps we define $S^{(d)}\in B(\Gamma)$ such that $G^{\{d\}}=\Cay(\Gamma,S^{(d)})$
is an integral Cayley graph over $\Gamma$. If $d$ is a distance not attained in $G$, 
then the assertion is confirmed by $G^{\{d\}}=\Cay(\Gamma,S^{(d)})$  with $S^{(d)}=\emptyset$. 
If $d=0$ then we achieve our goal by $S^{(0)}=\{0\}$.
Suppose now that $d=\infty$ and $G$ is disconnected. If $U=\left<S\right>$ is the
subgroup generated by $S$ in $\Gamma$, then $G$ consists of disjoint subgraphs on the cosets of $U$,
all of them isomorphic to $\Cay(U,S)$. Vertices $x,y$ in $G^{\{\infty\}}$ are adjacent if and only if
they belong to different cosets of $U$, and this is true if and only if $x-y\not\in U$. Therefore,
we have
$$G^{\{\infty\}} = \Cay(\Gamma,S^{(\infty)}) 
\text{ with } S^{(\infty)}=\overline{U}=\Gamma \backslash U\in B(\Gamma).$$ 
Assume now that $d\geq 1$ is a finite distance attained between vertices $x,y$ in $G$. The sequence
of vertices in a shortest path $P$ between $x$ and $y$ in $G=\Cay(\Gamma,S)$ has the form
$$x,x+s_1,x+s_1+s_2,\ldots,x+s_1+\ldots+s_d = y,~s_i\in S \text{ for }1\leq i\leq d.$$ 
This implies $y-x = s_1+\ldots +s_d \in dS$, where $dS$ denotes the $d$-fold sum of the set $S$.
To guarantee that there is no shorter path from $x$ to $y$ than $P$ we remove from $dS$ all multiples
$kS$ for $0\leq k<d$, $0S=\{0\}$. Setting
\begin{equation}\label{eq16}
S^{(d)}~=~dS~\backslash~\bigcup_{0\leq k<d}~kS
\end{equation}
we achieve  $G^{\{d\}}=\Cay(\Gamma,S^{(d)})$.
If $G=\Cay(\Gamma,S)$ is integral, then we have $S\in B(\Gamma)$ by Theorem \ref{thm2}, $kS\in B(\Gamma)$ for
every $k\geq 2$ by Theorem \ref{thm1}, and trivially $0S=\{0\}\in B(\Gamma)$. By (\ref{eq16}) this implies 
$S^{(d)}\in B(\Gamma)$, so  $G^{\{d\}}$ is an integral Cayley graph over $\Gamma$.

To complete our proof, let 
$$D=\{d_1,\ldots,d_r\}\subseteq \{0,1,\ldots,\infty\}\text{ and } S^{(D)}=\bigcup_{i=1}^r S^{(d_i)}.$$
Then we have $S^{(D)}\in B(\Gamma)$ and $G^D=\Cay(\Gamma,S^{(D)})$ is an integral Cayley graph 
over $\Gamma$ by Theorem \ref{thm2}.
\end{proof}

Now we define a generalized distance matrix $\DM(k,G)$ of a given undirected graph $G$ with 
vertices $v_1,\ldots,v_n$ as follows. Let $d_0=0<d_1<\ldots <d_r$ be the sequence of
possible distances between vertices in $G$, possibly $d_r=\infty$. If $k=(k_0,\ldots,k_r)$ is a
vector with integral entries, then we define the entries of $\DM(k,G)=(d_{i,j}^{(k)})$ 
for $i,j\in\{1,\ldots,n\}$ by 
$$
d_{i,j}^{(k)} ~=~ k_t \text{~, if }~d(v_i,v_j)=d_t.
$$
The ordinary distance matrix $\DM(G)$ for a connected graph $G$ is established for $k=(0,1,...,r)$,
where $r$ is the diameter of $G$.

Let $\Gamma=\{v_1,\ldots,v_n\}$ be an abelian group and consider some integral Cayley graph $G=\Cay(\Gamma,S)$. 
Any generalized distance matrix $\DM(k,G)$ is an
integer weighted sum of the adjacency matrices of the graphs $G^{\{d\}}$ with $d\in\{d_0,d_1,\ldots,d_r\}$,
assuming $v_1,\ldots,v_n$ as their common vertex order.
To make it more precise, for $j=0,\ldots,r$ we denote by $A^{(j)}$ the adjacency matrix of the distance power
$G^{\{d_j\}}$, $A^{(0)}=I_n$ is the $n\times n$ unit matrix. Then we have
$$ \DM(k,G) = k_0A^{(0)}+k_1A^{(1)}+\ldots +k_rA^{(r)}.$$
By Theorem \ref{thm3}, all matrices $A^{(j)}$, $0\leq j\leq r$, are adjacency matrices of integral Cayley
graphs over $\Gamma$. According to Proposition \ref{prop3}, all Cayley graphs over $\Gamma$ have a universal
common basis of complex eigenvectors. As a result, integrality extends to $\DM(k,G)$.
This proves the following theorem.

\begin{theorem}
Let $G=\Cay(\Gamma,S)$ be an integral Cayley graph over the abelian group $\Gamma$, $|\Gamma|=n$.
Then every distance matrix $\DM(k,G)$ as defined above has integral spectrum.
Moreover, the characters $\psi_1,\ldots,\psi_n$ of $\Gamma$ represent an orthogonal basis of $\C^n$
consisting of eigenvectors of $\DM(k,G)$.
\end{theorem}

As we have seen in Theorem \ref{thm3}, the class of integral Cayley graphs over an abelian group is 
closed under distance power operations. We shall conclude this section by presenting a subclass
which has the same closure property. 

We introduce the class of {\em gcd-graphs} as in \cite{klotz2}. To this end, let the finite abelian group $\Gamma$
be represented as $\Gamma=\Z_{m_1}\oplus\ldots\oplus\Z_{m_r}$, $m_i\geq 1$ for $i=1,\ldots,r$, cf. (\ref{eq8}). 
Hence the elements $x\in \Gamma$ take the form of $r$-tuples. 
For $x=(x_1,\ldots,x_r)\in \Gamma$ and $m=(m_1,\ldots,m_r)$ we define 
$$  \gcd(x,m) = (\gcd(x_1,m_1),\ldots,\gcd(x_r,m_r)).$$
Here we agree upon $gcd(0,m_i)=m_i$. For a divisor tuple $d=(d_1,\ldots, d_r)$ of $m$, $d\mid m$, we require
$d_i\geq 1$ and $d_i\mid m_i$ for $i=1,\ldots, r$.
Every divisor tuple $d$ of $m$ defines an {\em elementary gcd-set} given by
$$
  S_\Gamma(d) = \{ x\in\Gamma:~\gcd(x,m)=d\}.
$$
Clearly, the sets $S_\Gamma(d)$ with $d\mid m$ form a partition of the elements of $\Gamma$.
We denote by $E_\Gamma(x)$ the unique elementary gcd-set that contains $x$, i.e.~$E_\Gamma(x)=S_\Gamma(d)$
with $d=\gcd(x,m)$.
A {\em gcd-set} is a union of elementary gcd-sets. By construction, the elementary gcd-sets are 
the atoms of the Boolean algebra $B_{\gcd}(\Gamma)$ consisting of all gcd-sets of $\Gamma$. 
According to Theorem~1 in \cite{klotz2}, $B_{\gcd}(\Gamma)$ is a Boolean sub-algebra of
$B(\Gamma)$. Hence by Theorem \ref{thm2}, all gcd-graphs $Cay(\Gamma,S)$, $S\in B_{gcd}(\Gamma)$, are integral.

\begin{lemma}\label{lem:esum}
If $\Gamma=\Z_{m_1}\oplus\ldots\oplus\Z_{m_r}$ and $x=(x_1,\ldots,x_r)\in\Gamma$ then
$$ E_\Gamma(x)=E_{\Z_{m_1}}(x_1)\times\ldots\times E_{\Z_{m_r}}(x_r).$$
\end{lemma}

\begin{proof}
Let $m=(m_1,\ldots,m_r)$ and $d=(d_1,\ldots,d_r)=\gcd(x,m)$.
Then we have $y=(y_1,\ldots,y_r)\in E_\Gamma(x)$ if and only if $\gcd(y_i,m_i)=d_i$ for $i=1,\ldots,r$.
This is equivalent to $y\in S_{\Z_{m_1}}(d_1)\times\ldots\times S_{\Z_{m_r}}(d_r)$, which is the 
same as $y\in E_{\Z_{m_1}}(x_1)\times\ldots\times E_{\Z_{m_r}}(x_r)$.
\end{proof}

\begin{lemma}\label{lem:gcdsetsum}
For every finite abelian group $\Gamma$, any sum of its gcd-sets is again a gcd-set.
\end{lemma}

\begin{proof}
As in the proof of Theorem \ref{thm1} it suffices to show that any sum of elementary gcd-sets is a gcd-set.
If $\Gamma$ is cyclic, then $B_{gcd}(\Gamma)=B(\Gamma)$ (see Theorem 3 in \cite{klotz2}) and the result follows from 
Lemma \ref{lem:atomsum}.

Now let $\Gamma=\Z_{m_1}\oplus\ldots\oplus\Z_{m_r}$, $m=(m_1,\ldots,m_r)$, $r\geq 2$. Further let 
$x=(x_1,\ldots,x_r)\in \Gamma$, $\gcd(x,m)=d=(d_1,\ldots,d_r)$ and 
let $y=(y_1,\ldots,y_r)\in\Gamma$, $\gcd(y,m)=\delta=(\delta_1,\ldots,\delta_r)$. 
By Lemma \ref{lem:esum} we have
$$  E_\Gamma(x)+E_\Gamma(y) = (E_{\Z_{m_1}}(x_1)+E_{\Z_{m_1}}(y_1))\times\ldots\times(E_{\Z_{m_r}}(x_r)+E_{\Z_{m_r}}(y_r)).$$
Since the cyclic case is already solved, it follows that $E_{\Z_{m_i}}(x_i)+E_{\Z_{m_i}}(y_i)$ is a gcd-set of $\Z_{m_i}$ 
for $i=1,\ldots,r$. Hence $E_{\Z_{m_i}}(x_i)+E_{\Z_{m_i}}(y_i)$ is a disjoint union of elementary
gcd-sets $E_{\Z_{m_i}}(z^{(i)}_1),\ldots,E_{\Z_{m_i}}(z^{(i)}_{\varrho_i})$, with $z^{(i)}_j\in\Z_{m_i}$ for $j=1,\ldots,\varrho_i$.
It follows that
\[
  E_\Gamma(x)+E_\Gamma(y)=\bigcup\limits_{1\leq j_k \leq \varrho_k,~k=1,\ldots,r} 
\left( E_{\Z_{m_1}}(z^{(1)}_{j_1}) \times \ldots \times E_{\Z_{m_r}}(z^{(r)}_{j_r}) \right).
\]
Writing $z^{(j_1,\ldots,j_r)}=(z^{(1)}_{j_1},\ldots,z^{(r)}_{j_r})$, we get by Lemma \ref{lem:esum}
$$ E_\Gamma(x)+E_\Gamma(y)=\bigcup\limits_{1\leq j_k\leq \varrho_k,~k=1,\ldots,r} 
E_\Gamma(z^{(j_1,\ldots,j_r)}) ~\in B_{gcd}(\Gamma).$$ 
\end{proof}

The following theorem is readily deduced from Lemma \ref{lem:gcdsetsum} 
applying the same reasoning as in the proof of Theorem \ref{thm3}.

\begin{theorem}\label{thm:gcddist}
If $G=\Cay(\Gamma,S)$ is a gcd-graph over $\Gamma = \Z_{m_1}\oplus\ldots\oplus\Z_{m_r}$
and if $D$ is a set of nonnegative integers (possibly including $\infty$), 
then the distance power $G^D$ is also a gcd-graph over $\Gamma$. 
\end{theorem}

%-------------------------------------------------------------------------------------------------------

\end{document}